\documentclass[11pt,a4paper]{article}
\usepackage{epsf,epsfig,amsfonts,amsgen,amsmath,amssymb,amstext,amsbsy,amsopn,amsthm,lineno}
\usepackage{color}
\usepackage{graphicx}
\usepackage{caption}
\newtheorem{theorem}{Theorem}

\newtheorem{lemma}{Lemma}

\newtheorem{fact}{Fact}

\theoremstyle{definition}
\newtheorem{definition}{Definition}

\newtheorem{claim}{Claim}

\newtheorem{problem}{Problem}
\newtheorem{case}{Case}

\begin{document}
\title
{\bf\Large Cyclability of {\em id}-cycles in graphs\thanks{Supported by NSFC (No.~11271300).}}

\date{}
\author{\small Ruonan Li${}^1$, Bo Ning${}^2$\thanks{E-mail address: bo.ning@tju.edu.cn (B. Ning)}, Shenggui Zhang${}^1$\thanks{Corresponding author. E-mail address: sgzhang@nwpu.edu.cn (S. Zhang)}
 \\[2mm]
\small ${}^1$Department of Applied Mathematics, Northwestern Polytechnical University,\\
\small Xi'an, Shaanxi 710072, P.R.~China\\
\small ${}^2$Center for Applied Mathematics, Tianjin University,\\
\small Tianjin 300072, P.R.~China}
\maketitle

\begin{abstract}
 Let $G$ be a graph on $n$ vertices and $C'=v_0v_1\cdots v_{p-1}v_0$ a vertex sequence of $G$ with $p\geq 3$ ($v_i\neq v_j$ for all $i,j=0,1,\ldots,p-1$, $i\neq j$). If for any successive vertices $v_i$, $v_{i+1}$ on $C'$, either $v_iv_{i+1}\in E(G)$ or both of the first implicit-degrees of $v_i$ and $v_{i+1}$ are at least $n/2$ (indices are taken modulo $p$), then $C'$ is called an $id$-cycle of $G$. In this paper, we prove that for every $id$-cycle $C'$, there exists a cycle $C$ in $G$ with $V(C')\subseteq V(C)$. This generalizes several early results on the Hamiltonicity and {cyclability} of graphs.
\medskip

\smallskip

\end{abstract}

\section{Introduction}

All graphs considered in this paper are finite, simple and undirected. For terminology and notation not defined here we refer the reader to \cite{Bondy:76}.

Let $G$ be a graph. The vertex set and {edge} set of $G$ are denoted by $V(G)$ and $E(G)$, respectively. For a vertex $v$ and a subgraph $H$ of $G$, the {\em neighborhood} of $v$ in $H$ is defined as $N_H(v)=\{u: u\in V(H), uv\in E(G)\}$ and the {\em degree} of $v$ in $H$ is defined as $d_H(v)=|N_H(v)|$. If there is no ambiguity, we write $N(v)$ for $N_G(v)$ and $d(v)$ for $d_G(v)$.

In the study of the existence of Hamilton cycles in graphs, degree conditions play very important roles. Among the many results of this direction, the following {two} are well known.

\begin{theorem}[Dirac \cite{Dirac:52}]
\label{Dirac}
Let $G$ be a graph on $n\geq 3$ vertices. If $d(v)\geq n/2$ for every vertex $v\in V(G)$, then $G$ is Hamiltonian.
\end{theorem}

\begin{theorem}[Ore \cite{O.Ore:60}]
\label{Ore}
Let $G$ be a graph on $n\geq 3$ vertices. If $d(u) + d(v)\geq n$ for every pair of nonadjacent vertices $u,v\in V(G)$, then $G$ is Hamiltonian.
\end{theorem}

For a vertex $v$ of a graph $G$, denote by $N_2(v)$ the vertices which are at distance of 2 from $v$ in $G$. In order to weaken the condition in Theorem \ref{Ore} (Ore's condition), Zhu et al. \cite{Y.Zhu:89} gave the definition of the first implicit-degree of the vertex $v$ based on the degrees of vertices in $N(v)\cup N_2(v)\cup \{v\}$.

\begin{definition}[Zhu et al. \cite{Y.Zhu:89}]
\label{Def:implicit}
Let $G$ be a graph on $n$ vertices and $v$ a vertex in $G$. If $N_2(v)\neq \emptyset$, then let $d(v)=k+1$, $M_2=\max\{d(u)|u\in N_2(v)\}$. Denote by $d_1\leq d_2\leq \cdots \leq d_{k+1}\leq d_{k+2}\leq \cdots$ the nondecreasing degree sequence of vertices in $N(v)\cup N_2(v)$. Then the first implicit-degree of $v$ is defined as
\begin{equation}
d_1(v)=
\begin{cases}
\max\{d_{k+1}, k+1\},&\text{if~$d_{k+1}>M_2$};\cr
\max\{d_k, k+1\},&\text{otherwise}.\cr
\end{cases}
\end{equation}
If $N_2(v)=\emptyset$, then $d(v)=n-1$. In this case, let $d_1(v)=d(v)=n-1$.
\end{definition}

It is clear that $d_1(v)\geq d(v)$ for every vertex $v\in V(G)$. Zhu et al. \cite{Y.Zhu:89} obtained the following result as a generalization of Theorem \ref{Ore}.

\begin{theorem}[Zhu et al. \cite{Y.Zhu:89}]
\label{Zhu}
Let $G$ be a $2$-connected graph on {$n\geq 3$} vertices. If $d_1(u) + d_1(v)\geq n$ for every pair of nonadjacent vertices $u,v\in V(G)$, then $G$ is Hamiltonian.
\end{theorem}

Let $G$ be a graph and $X$ a subset of of $V(G)$. If there exists a cycle $C$ in $G$ with $X\subseteq V(C)$, then we say $X$ is {\it cyclable} in $G$. A subgraph $H$ of $G$ is called {\em cyclable} if $V(H)$ is cyclable. Apparently, $G$ is Hamiltonian if and only if every spanning subgraph of $G$ is cyclable.

For a graph $G$, a vertex of degree at least $|V(G)|/2$ is called {\it heavy}.
In 1992, Shi \cite{R.Shi:92} proved the following result.

\begin{theorem}[Shi \cite{R.Shi:92}]
\label{Shi}
Let $G$ be a $2$-connected graph on {$n\geq 3$} vertices and $S=\{v:d(v)\geq n/2,v\in V(G)\}$. Then $S$ is cyclable in $G$.
\end{theorem}

It is clear that Theorem \ref{Shi} implies Theorems \ref{Dirac} and \ref{Ore}.

Recently Li et al. \cite{B.Li:12} gave another generalization of Ore's condition.

\begin{definition}[Li et al. \cite{B.Li:12}]
\label{OCycle}
Let $G$ be a graph on $n$ vertices and $C'=v_0v_1\cdots v_{p-1}v_0$ a vertex sequence in $G$ with $p\geq 3$ ($v_i\neq v_j$ for all $i,j=0,1,\ldots,p-1$, $i\neq j$). If for any successive vertices $v_i$, $v_{i+1}$ on $C'$, either $v_iv_{i+1}\in E(G)$ or $d(v_i)+d(v_{i+1})\geq n$ (indices are taken modulo $p$), then $C'$ is called an {\em Ore-cycle} of $G$ or briefly, an {\em $o$-cycle} of $G$.
\end{definition}

\begin{theorem}[Li et al. \cite{B.Li:12}]
\label{BinLong}
Let {$C'$ an $o$-cycle of a graph $G$}. Then $C'$ is cyclable in $G$.
\end{theorem}

Obviously, Theorem \ref{BinLong} implies Theorems \ref{Ore} and \ref{Shi}. Our aim in this paper is to consider whether Theorem \ref{BinLong} can be generalized to the first implicit-degree condition.

\begin{definition}
\label{OCI}
Let $G$ be a graph on $n$ vertices and $C'=v_0v_1\cdots v_{p-1}v_0$ a vertex sequence in $G$ with $p\geq 3$ ($v_i\neq v_j$ for all $i,j=0,1,\ldots,p-1$, $i\neq j$). If for any successive vertices $v_i$, $v_{i+1}$ on $C'$, either $v_iv_{i+1}\in E(G)$ or $d_1(v_i)+d_1(v_{i+1})\geq n$ (indices are taken modulo $p$), then we call $C'$ an {\it implicit-Ore-cycle} of $G$ or briefly, an {\it $io$-cycle} of $G$. Specifically, if either  $v_iv_{i+1}\in E(G)$ or $d_1(v_i)\geq n/2$ and $d_1(v_{i+1})\geq n/2$ for all $i=0,1,\ldots,p-1$, then we call $C'$ an {\it implicit-Dirac-cycle} of $G$ or briefly, an {\it $id$-cycle} of $G$.
\end{definition}

\begin{problem}
\label{Pro:ico}
Is every $io$-cycle cyclable ?
\end{problem}

Although we are unable to solve Problem \ref{Pro:ico}, we can show that every $id$-cycle is cyclable.

\begin{theorem}
\label{Thm:icd}
Let {$C'$ an $id$-cycle of a graph $G$}. Then $C'$ is cyclable in $G$.
\end{theorem}

Note that in Definition \ref{OCycle}, if either $v_iv_{i+1}\in E(G)$ or $d(v_i)\geq n/2$ and $d(v_{i+1})\geq n/2$ for all $i=0,1,\ldots,p-1$, then we can similarly call $C'$ a {\it Dirac-cycle} of $G$ or briefly, a {\it $d$-cycle} of $G$. It is clear that a $d$-cycle is a special $o$-cycle and we can regard an $id$-cycle as a generalization of a $d$-cycle.

\begin{fact}
\label{rem:1}
Theorem \ref{Thm:icd} implies Theorem \ref{Zhu}.
\end{fact}

\begin{proof}
Let $A=\{a:d_1(a)<n/2, a\in V(G)\}$. For any vertices $u,v\in A$ ($u\neq v$), we have $d_1(u)+d_1(v)<n$. So $uv\in E(G)$. This implies that $G[A]\cong K_{|A|}$. Let $B=\{b:d_1(b)\geq n/2, b\in V(G)\}$. Now, consider the size of $|A|$ and $|B|$, respectively.

If $|B|=0$ or $|A|=0$, then any vertex sequence of length $n$ is an $id$-cycle in $G$. By Theorem \ref{Thm:icd}, $G$ is Hamiltonian.

If $|A|=1$, then let $A=\{a\}$ and $B=\{b_1,b_2,\ldots, b_{n-1}\}$. Since $G$ is $2$-connected, there are at least two neighbors of $a$ in $B$, say $b_1$ and $b_2$. Thus $b_1ab_2b_3\cdots b_{n-1}b_1$ is an $id$-cycle of length $n$. By Theorem \ref{Thm:icd}, $G$ is Hamiltonian. The proof is similar when $|B|=1$.

If $|A|\geq 2$ and $|B|\geq 2$, since $G$ is $2$-connected, there exist $a_1, a_2\in A$~$(a_1\neq a_2)$ and $b_1,b_2\in B$~$(b_1\neq b_2)$ such that $a_1b_1, a_2b_2\in E(G)$. Let $A=\{a_1,a_2,\ldots, a_k\}$ and $B=\{b_1,b_2,\ldots, b_{n-k}\}$. Then $b_1a_1a_3a_4\cdots a_ka_2b_2b_3\cdots b_{n-k}b_1$ is an $id$-cycle of length $n$ in $G$. By Theorem \ref{Thm:icd}, $G$ is Hamiltonian.
\end{proof}

\begin{fact}
\label{rem:2}
Let $G$ be a $2$-connected graph on {$n\geq 3$} vertices and $S=\{v:d_1(v)\geq n/2,v\in V(G)\}$. Then $S$ is cyclable in $G$.
\end{fact}

Obviously, Fact \ref{rem:2} is a generalization of Theorem \ref{Shi} and can be directly obtained from Theorem \ref{Thm:icd}.

\section{Definitions and Lemmas}

In this section, we will give some additional definitions and useful lemmas.

Let $G$ be a graph and $C'=v_0v_1\cdots v_{p-1}v_0$~($p\geq 3$) an $id$-cycle in $G$ with a fixed orientation. For vertices $x,y\in V(C')$, let $xC'y$ be the segment on $C'$ from $x$ to $y$ along the direction of $C'$ and $x\overline{C'}y$ the segment on $C'$ along the reverse direction. For a vertex $v_i\in V(C')$, if $v_{i-1}v_i$ or $v_iv_{i+1}\not\in E(G)$, then we call $v_i$ a {\it break-vertex} on $C'$. Denote by $Bre(C')$ the set of break-vertices on $C'$. Let
$$Bre^+(C')=\{v_i:v_iv_{i+1}\not\in E(G)\}~\text{and}~Bre^-(C')=\{v_i:v_{i-1}v_i\not\in E(G)\}.$$
Then $Bre(C')=Bre^+(C')\cup Bre^-(C')$. Note that $Bre^+(C')\cap Bre^-(C')$ is not necessarily empty.
For a vertex $v_i\in V(C')$, let $v^+_i=v_{i+1}$ and $v^-_i=v_{i-1}$. Then $v^+_i$ and $v^-_i$ represent the immediate successor and predecessor of $v_i$ on $C'$, respectively. Denote by $N_{C'}(v_i)^-$
the predecessors of vertices in $N_{C'}(v_i)$. To measure the gap between $C'$ and a cycle, we define the {\it deficit-degree} of $C'$ as
$$def(C')=|\{i:v_iv_{i+1}\not\in E(G)\}|.$$
If $def(C')\leq def(C)$ for any $id$-cycle $C$ satisfying $V(C')\subseteq V(C)$, then we say $C'$ is {\it def-minimal}.
Let $u$ be a break-vertex on $C'$. We say $u$ is a {\it heavy-break-vertex} if $d(u)\geq |V(G)|/2$. Denote by $Hb(C')$ the set of heavy-break-vertices on $C'$. To measure the difference between $C'$ and a $d$-cycle, we define the {\it heavy-index} of $C'$ as
$$hb(C')=|Hb(C')\cap Bre^+(C')|+|Hb(C')\cap Bre^-(C')|.$$
If $hb(C')\geq hb(C)$ for any $id$-cycle $C$ satisfying $V(C')\subseteq V(C)$ and $def(C')=def(C)$, then we say $C'$ is {\it hb-maximal}.

Let $P=u_0u_1\cdots u_{t-1}$ be a path in $G$. Then we call $u_0$ and $u_{t-1}$ the {\it end-vertices} of $P$. For vertices $a,b\in V(P)$, denote by $aPb$ the segment on $P$ from $a$ to $b$. If $a=b$, then $aPb=\{a\}$. Apparently, an $id$-cycle $C'$ in $G$ is composed of some vertex-disjoint paths and we can write $C'=x_1P_1y_1x_2P_2y_2\cdots x_sP_sy_sx_1$, where $x_i$ and $y_i$ are the end-vertices of $P_i$ satisfying $d_1(x_i)\geq |V(G)|/2$ and $d_1(y_i)\geq |V(G)|/2$ for all $i=1,2,\ldots,s$. Hence, the set of break-vertices on $C'$ can be regarded as the set of end-vertices of $P_i$~($i=1,2\ldots,s$).

Let $C'=v_0v_1\cdots v_{p-1}v_0$ be an $id$-cycle in a graph $G$. Then we have $def(C')\geq 0$ and $hb(C')\leq 2def(C')$. If $def(C')=0$, then $C'$ is a cycle. If $hb(C')=2def(C')$, then $C'$ is a $d$-cycle and cyclable. In this paper, we mainly consider the case that $def(C')>0$ and $hb(C')<2def(C')$. In order to make the paper easy to follow, we name a specific kind of break-vertex as ``strange-vertex''.

\begin{definition}
\label{strange}
Let $G$ be a graph on $n\geq3$ vertices and $C'=x_1P_1y_1x_2P_2y_2\cdots x_sP_sy_sx_1$ an $id$-cycle in $G$. Let $R$ be the subgraph of $G$ induced by $V(G)\backslash V(C')$ and $u$ an end-vertex of $P_i$. If the following conditions hold: \\
$(a)$ $d(u)< n/2$;\\
$(b)$ $d(v)<d_1(u)$ for every vertex $v\in N_R(u)$;\\
$(c)$ $N(u)\cap V(P_j)=\emptyset$ ($j=1,2,\ldots,s, j\neq i$);\\
$(d)$ $|V(P_i)|\geq 3$ and $uw\in E(G)$ ($w$ is the other end-vertex of $P_i$),\\
then we call $u$ a {\it strange-vertex} on $C'$. Denote by $Str(C')$ the set of strange-vertices on $C'$.
\end{definition}
\setcounter{figure}{0}

\begin{lemma}
\label{nan}
Let $G$ be a graph on $n\geq3$ vertices and $C'=v_0v_1\cdots v_{p-1}v_0$ an $id$-cycle in $G$. If $v_0v_{p-1}\not\in E(G)$ and $d(v_0)+d(v_{p-1})\geq n$, then there exists an $id$-cycle $C$ such that $V(C')\subseteq V(C)$ and $def(C)< def(C')$.
\end{lemma}

\begin{figure}[h]
  \begin{center}
    \includegraphics[width=0.40\textwidth]{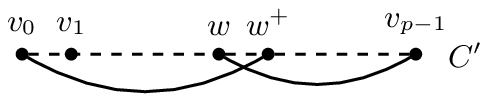}
  \end{center}
  \caption{\label{sum n}}
\end{figure}

\begin{proof}
Let $R$ be the subgraph of $G$ induced by $V(G)\backslash V(C')$. If there exists a vertex $w\in N_R(v_0)\cap N_R(v_{p-1})$, then construct a new $id$-cycle as $C=wv_0v_1\cdots v_{p-1}w$. Obviously, $def(C)=def(C')-1<def(C')$.  If $N_R(v_0)\cap N_R(v_{p-1})=\emptyset$, then $|N_R(v_0)|+|N_R(v_{p-1})|\leq|R|$. Since $d(v_0)+d(v_{p-1})\geq n$, we have $|N_{C'}(v_0)|+|N_{C'}(v_{p-1})|\geq|C'|$. Note that $|N_{C'}(v_0)^-|=|N_{C'}(v_0)|$, so
$$|N_{C'}(v_0)^-|+|N_{C'}(v_{p-1})|\geq|C'|.$$
Since $v_0v_{p-1}\not\in E(G)$, we have
$$|N_{C'}(v_0)^-\cup N_{C'}(v_{p-1})|\leq |C'|-1.$$
This implies that $N_{C'}(v_0)^-\cap N_{C'}(v_{p-1})\neq\emptyset$. Choose a vertex $w\in N_{C'}(v_0)^-\cap N_{C'}(v_{p-1})$, then $w^+\in N_{C'}(v_0)$. Construct an $id$-cycle as $C=v_0w^+C'v_{p-1}w\overline{C'}v_0$ (see Fig. \ref{sum n}). Apparently, $V(C')\subseteq V(C)$ and $def(C)<def(C')$.
\end{proof}

\begin{lemma}
\label{Li}
Let $G$ be a graph on $n\geq3$ vertices and $C'=v_0v_1\cdots v_{p-1}v_0$ a def-minimal $id$-cycle in $G$ with $v_0v_{p-1}\not\in E(G)$. Let $R$ be the subgraph of $G$ induced by $V(G)\backslash V(C')$. If $v_0$ satisfies the following conditions:\\
$(a)$ $N_{C'}(v_0)^-\subseteq N(v_0)\cup N_2(v_0)\cup\{v_0\}$;\\
$(b)$ $N_2(v_0)\nsubseteq N_{C'}(v_0)^-$;\\
$(c)$ $d(v_0)<n/2$ and $d(v)<d_1(v_0)$ for any $v\in N_R(v_0)$,\\
then there must exist a vertex $u\in N_{C'}(v_0)^-$ such that $d(u)\geq d_1(v_0)$ and $C'$ is not hb-maximal.\\
\end{lemma}

\begin{proof}
Suppose that $d(v_0)=k+1$. Denote by $d_1\leq d_2\leq \cdots \leq d_{k+1}\leq d_{k+2}\leq \cdots$ the non-decreasing degree sequence of $N(v_0)\cup N_2(v_0)$. Let $M_2=\max\{d(u)|u\in N_2(v_0)\}$. Since $N_{C'}(v_0)^-\subseteq {V(C')}$ and $N_R(v_0)\subseteq V(R)$, we have $N_{C'}(v_0)^-\cap N_R(v_0)=\emptyset$. Furthermore, $|N_{C'}(v_0)^-|=|N_{C'}(v_0)|$, so $|N_{C'}(v_0)^-\backslash\{v_0\}\cup N_R(v_0)|\geq d(v_0)-1$ (the equation holds if and only if $v_0v_1\in E(G)$). Thus we get
\begin{equation}
\label{geshuweiK}
|N_{C'}(v_0)^-\backslash\{v_0\}\cup N_R(v_0)|\geq k.
\end{equation}
By $(a)$, we have
\begin{equation}
\label{baohan}
N_{C'}(v_0)^-\backslash\{v_0\}\cup N_R(v_0)\subseteq N(v_0)\cup N_2(v_0).
\end{equation}
Since $d_1(v_0)\geq n/2 > d(v_0)$, $d_1(v_0)=d_k$ or $d_1(v_0)=d_{k+1}$.

If $d_1(v_0)=d_k$, then there are at most $k-1$ vertices in $N(v_0)\cup N_2(v_0)$ having degrees smaller than $d_1(v_0)$. By (\ref{geshuweiK}) and (\ref{baohan}),  there exists a vertex $u\in N_{C'}(v_0)^-\backslash\{v_0\}\cup N_R(v_0)$ such that $d(u)\geq d_1(v_0)$.

If $d_1(v_0)=d_{k+1}>d_k$, then by Definition \ref{Def:implicit}, $d_{k+1}>M_2$. By ($b$), there is a vertex $w\in N_2(v_0)$ and $w\not\in N_{C'}(v_0)^-\backslash\{v_0\}\cup N_R(v_0)$ satisfying $d(w)\leq M_2 < d_{k+1}$. Similarly, by (\ref{geshuweiK}) and (\ref{baohan}), there exists at least one vertex $u\in N_{C'}(v_0)^-\backslash\{v_0\}\cup N_R(v_0)$ such that $d(u)\geq d_{k+1}=d_1(v_0)$.

Recall that $d(v)<d_1(v_0)$ for any $v\in N_R(v_0)$. In all cases, there is a vertex $u\in N_{C'}(v_0)^-$ satisfying $d(u)\geq d_1(v_0)$.

Let $u=v_s$ and $C=v_sv_{s-1}$ $\cdots v_0v_{s+1}v_{s+2}\cdots v_{p-1}$. Thus $V(C)=V(C')$ and $def(C)\leq def(C')$. Note that $C'$ is def-minimal, we have $def(C)=def(C')$.

Now, we will prove that $hb(C)>hb(C')$. Considering the construction of $C$, we know that
$$Bre^+(C)=\{v_t: v_t\in Bre^-(C'),t\leq s\}\cup\{v_t: v_t\in Bre^+(C'), t>s\},$$
and
$$Bre^-(C)=\{v_t: v_t\in Bre^+(C'),t\leq s\}\cup\{v_t: v_t\in Bre^-(C'), t>s\}\cup\{v_s\}\backslash\{v_0\}.$$
Thus we have
\begin{eqnarray*}
hb(C)&=&|Hb(C)\cap Bre^+(C)\}|+|Hb(C)\cap Bre^-(C)\}|\\
&=&|\{v_t:v_t\in Hb(C')\cap Bre^-(C'),t\leq s \}|\\
& &+|\{v_t:v_t\in Hb(C')\cap Bre^+(C'),t>s \}|\\
& &+|\{v_t:v_t\in Hb(C')\cap Bre^+(C'),t\leq s \}|\\
& &+|\{v_t:v_t\in Hb(C')\cap Bre^-(C'),t>s \}|+|\{v_s\}|\\
&=&hb(C')+1.
\end{eqnarray*}
Hence, $C'$ is not hb-maximal. The proof is complete.
\end{proof}

\begin{lemma}
\label{duandian}
Let $G$ be a graph on $n\geq3$ vertices and $C'=x_1P_1y_1x_2P_2y_2\cdots x_sP_sy_sx_1$ a def-minimal and then hb-maximal $id$-cycle in $G$ with $def(C')\geq 1$. Then the following statements hold:\\
$(1)$~$x_ix_j, x_iy_j, y_iy_j\not\in E(G)$ for any $i,j=1,2,\ldots,s, i\neq j$;\\
$(2)$~$Bre(C')=Hb(C')\cup Str(C')$.
\end{lemma}

\begin{proof}
$(1)$ By contradiction. Assume that $x_ix_j\in E(G)$. Combine $P_i$ and $P_j$ into a new path $P'=y_iP_ix_ix_jP_jy_j$. Note that although we change the orders or orientations of $P_i$ and $P_j$ in $C'$, it always produces an $id$-cycle. We can assume that $C$ is an arbitrary permutation of $\{P_1,P_2,\ldots,P_s\}\backslash\{P_i,P_j\}\cup \{P'\}$. Thus $def(C)\leq def(C')-1<def(C')$. This contradicts that $C'$ is a def-minimal. Similarly, we can prove that $x_iy_j, y_iy_j\not\in E(G)$.

$(2)$ By contradiction. Assume that there is a break-vertex $x_i$ which is neither a strange-vertex nor a heavy-break-vertex. Then $d(x_i)<n/2$. Let $R$ be the subgraph of $G$ induced by $V(G)\backslash V(C')$. By Definition \ref{strange}, at least one of the following statements fails:\\
$(a)$ $d(v)<d_1(x_i)$ for every vertex $v\in N_R(x_i)$;\\
$(b)$ $N(x_i)\cap V(P_j)=\emptyset$ for any $j=1,2,\ldots,s$, $j\neq i$;\\
$(c)$ $|V(P_i)|\geq 3$ and $x_iy_i\in E(G)$.\\
Without loss of generality, let $C'=x_iP_iy_ix_{i+1}P_{i+1}y_{i+1}\cdots x_{i-1}P_{i-1}y_{i-1}x_i$. Denote by $v_0v_1\cdots v_{p-1}v_0$ the vertex sequence of $C'$ with $v_0=x_i$.  Let $l(x_i)=\max\{t|v_tv_0\in E(G)\}$. By $(1)$, we have $N_{C'}(x_i)^-\subseteq N(x_i)\cup N_2(x_i)\cup\{x_i\}$.

Now, we will discuss the following three cases.

\begin{case}
$(a)$ fails.
\end{case}

In this case, there is a vertex $v\in N_R(x_i)$ such that $d(v)\geq d_1(x_i)\geq n/2$. Let $C=vv_0v_1\cdots v_{p-1}v$. Then $V(C')\subseteq V(C)$ and $def(C)\leq def(C')$. Since $C'$ is a def-minimal $id$-cycle, we have $def(C)=def(C')$ and $vv_{p-1}\not\in E(G)$. Thus $Bre^+(C)=Bre^+(C')$, $Bre^-(C)=Bre^-(C')\cup\{v\}\backslash\{x_i\}$ and $hb(C)=hb(C')+1$. This contradicts that $C'$ is hb-maximal.

\begin{case}
$(a)$ holds and $(b)$ fails.
\end{case}

In this case, there exists a path $P_j$ ($j=1,2,\ldots,s$, $j\neq i$) such that $N(x_i)\cap V(P_j)\neq\emptyset$. By $(1)$, $x_j,y_j\not\in N(x_i)$. Thus $v_{l(x_i)}v_{l(x_i)+1}\in E(G)$. This implies that $v_{l(x_i)+1}\in N_2(x_i)$. Since $ v_{l(x_i)+1}\not\in N_{C'}(x_i)^-$, we have $N_2(x_i)\nsubseteq N_{C'}(x_i)^-$. Thus, the vertex $x_i$ on the $id$-cycle $C'$ suffices the conditions in Lemma \ref{Li}. Hence, $C'$ is not hb-maximal, a contradiction.

\begin{case}
$(a)$, $(b)$ hold and $(c)$ fails.
\end{case}

In this case, $|V(P_i)|\leq 2$ or $|V(P_i)|\geq 3$ and $x_iy_i\not\in E(G)$. If $|V(P_i)|\leq 2$, then $N_{C'}(x_i)^-\subseteq \{x_i\}$. So $N_2(x_i)\cap N_{C'}(x_i)^-=\emptyset$. Since $N_2(x_i)\neq\emptyset$, we have $N_2(x_i)\nsubseteq N_{C'}(x_i)^-$. If $|V(P_i)|\geq 3$ and $x_iy_i\not\in E(G)$, then $N_{C'}(x_i)\neq \emptyset$ and $v_{l(x_i)+1}\in N_2(x_i)$. Furthermore, we know that $v_{l(x_i)+1}\not\in N_{C'}(x_i)^-$, so $N_2(x_i)\nsubseteq N_{C'}(x_i)^-$. No matter $|V(P_i)|\leq 2$ or $|V(P_i)|\geq 3$, the vertex $x_i$ on the $id$-cycle $C'$ suffices the conditions in Lemma \ref{Li}. Thus, $C'$ is not hb-maximal, a contradiction.

Now, each break-vertex $x_i$ on $C'$ is either is strange-vertex or a heavy-break-vertex. Similarly, we can prove this conclusion for every break-vertex $y_i$ by analyzing the reversion of $C'$.

The proof is complete.
\end{proof}

\begin{lemma}
\label{contain}
Let $G$ be a graph on $n\geq3$ vertices and $C'=x_1P_1y_1x_2P_2y_2\cdots x_sP_sy_sx_1$ a def-minimal and then hb-maximal $id$-cycle in $G$ with $def(C')\geq 1$. If $x_i\in Str(C')$, then $N_2(x_i)\subseteq N_{C'}(x_i)^-\subseteq V(P_i)$.
\end{lemma}

\begin{proof}
Without loss of generality, assume $C'=v_0v_1\cdots v_{p-1}v_0$ starts at $v_0=x_i$ with $v_0v_{p-1}\not\in E(G)$. First, we will prove that $N_2(x_i)\subseteq N_{C'}(x_i)^-$.

By contradiction. Assume that $N_2(x_i)\nsubseteq N_{C'}(x_i)^-$. Recall the definition of strange-vertex. We know that the vertex $x_i$ on the $id$-cycle $C'$ suffices the conditions of Lemma \ref{Li}. Thus, $C'$ is not hb-maximal, a contradiction.

Furthermore, by the definition of strange-vertex, we have $N_{C'}(x_i)^-\subseteq V(P_i)$.  So $N_2(x_i)\subseteq N_{C'}(x_i)^-\subseteq V(P_i)$.
\end{proof}

\setcounter{case}{0}
\setcounter{claim}{0}

\section{Proof of Theorem \ref{Thm:icd}}

By contradiction. Assume that $|V(G)|=n$. Let $C_1$ be a def-minimal and then hb-maximal counterexample with $def(C_1)\geq 1$. By Lemma \ref{duandian}, we have $Bre(C_1)=Str(C_1)\cup Hb(C_1)$.

\begin{claim}
\label{1}
$def(C_1)\geq 2$.
\end{claim}

\begin{proof}
Assume that $def(C_1)=1$. Then $C_1$ is a path in $G$. Let $C_1=v_0v_1\cdots v_{p-1}$ and $v_0v_{p-1}\not\in E(G)$. By the definition of strange-vertex, we have $v_0, v_{p-1}\not\in Str(C_1)$.  This implies that $d(v_0)\geq n/2$ and $d(v_{p-1})\geq n/2$. By Lemma~\ref{nan}, there must exist an $id$-cycle $C_2$ in $G$ such that $V(C_1)\subseteq V(C_2)$ and $def(C_2)< def(C_1)$, a contradiction.
\end{proof}

Now, let $C_1=x_1P_1y_1x_2P_2y_2\cdots x_sP_sy_sx_1$. By Claim \ref{1}, we have $s\geq 2$. Let $i,j$ be arbitrary integers satisfying $1\leq i<j\leq s$.

\begin{claim}
\label{ordernary}
$x_i\in Str(C_1)$ or $x_j\in Str(C_1)$.
\end{claim}

\begin{proof}
By contradiction. Assume that $x_i\in Hb(C_1)$ and $x_j\in Hb(C_1)$. Then by changing the orders and orientations of the paths in $C_1$  appropriately we can  construct a new $id$-cycle $C_2$ such that $x_i$ and $x_j$ are successive on $C_2$. Since $d(x_i)+d(x_j)\geq n$, by Lemma \ref{nan}, there exists an $id$-cycle $C_3$ satisfying $V(C_1)=V(C_2)\subseteq V(C_3)$ and $def(C_3)< def(C_2)=def(C_1)$, a contradiction.
\end{proof}

\begin{claim}
\label{special}
$x_i\in Hb(C_1)$ or $x_j\in Hb(C_1)$.
\end{claim}

\begin{proof}
By contradiction. Assume that $x_i\in Str(C_1)$ and $x_j\in Str(C_1)$. By the definitions of strange-vertex and implicit-degree, there must exist vertices $u_i \in V(P_i)\cap N(x_i)$ and $u_j \in V(P_j)\cap N(x_j)$ satisfying $d(u_i)\geq d_1(x_i)\geq n/2$ and $d(u_j)\geq d_1(x_j)\geq n/2$, respectively. Thus we have $d(u_i)+d(u_j)\geq n$. So either $u_iu_j\in E(G)$ or $N(u_i)\cap N(u_j)\neq\emptyset$.

\begin{case}
$u_iu_j\in E(G)$
\end{case}

In this case, $x_iu_iu_j$ is a shortest path from $x_i$ to $u_j$ in $G$. So $u_j\in N_2(x_i)$ and $N_2(x_i)\nsubseteq V(P_i)$. This contradicts to Lemma \ref{contain}.

\begin{case}
$u_iu_j\not\in E(G)$
\end{case}

In this case, there is a vertex $w\in N(u_i)\cap N(u_j)$.

If $w\in V(P_i)$~(or~$V(P_j)$), then it follows from the definition of strange-vertex that $w\in N_2(x_j)$ (or $N_2(x_i)$). So $N_2(x_j)\nsubseteq V(P_j)$ (or $N_2(x_i)\nsubseteq V(P_i)$). This contradicts to Lemma \ref{contain}.

If $w\in V(P_k)$ and $k\neq i,j$, then it follows from the definition of strange-vertex that $w\in N_2(x_i)$ and $N_2(x_i)\nsubseteq V(P_i)$. This contradicts to Lemma \ref{contain}.

If $w\in V(G)\backslash V(C_1)$, then consider the relation between $w$ and $x_i$. If $wx_i\in E(G)$, then $x_iwu_j$ is a shortest path from $x_i$ to $u_j$. Thus $u_j\in N_2(x_i)$. If $wx_i\not\in E(G)$, then $w\in N_2(x_i)$. So, in all cases, we have $N_2(x_i)\nsubseteq V(P_i)$. This contradicts to Lemma \ref{contain}.
\end{proof}

\begin{claim}
\label{2}
$def(C_1)=2$.
\end{claim}

\begin{proof}
By contradiction. Assume that $def(C_1)\neq 2$. By Claim \ref{1}, $def(C_1)\geq 3$. For any integers $i,j,k$ satisfying $1\leq i<j<k\leq s$, we have $|\{x_i,x_j,x_k\}\cap Str(C_1)|\geq 2$ or $|\{x_i,x_j,x_k\}\cap Hb(C_1)|\geq 2$. This contradicts to Claim \ref{ordernary} or Claim \ref{special}.
\end{proof}

Now, we can assume that $C_1=x_1P_1y_1x_2P_2y_2x_1$. Without loss of generality, let $x_1\in Str(C_1)$ and $x_2\in Hb(C_1)$. By the definitions of strange-vertex and implicit-degree, there must exist an vertex $u\in V(P_1)\cap N(x_1)$ such that $d(u)\geq d_1(x_1)\geq  n/2$. Since $d(x_2)\geq n/2$, we have $ux_2\in E(G)$ or $N(u)\cap N(x_2)\neq\emptyset$.

If $ux_2\in E(G)$, then $x_2\in N_2(x_1)$ and $N_2(x_1)\nsubseteq V(P_1)$. This contradicts to Lemma \ref{contain}. So there exists a vertex $w\in N(u)\cap N(x_2)$.

If $w\in V(P_2)$, then $w\in N_2(x_1)$ and $N_2(x_1)\nsubseteq V(P_1)$, a contradiction. If $w\in V(G)\backslash V(C_1)$, then consider the relation between $w$ and $x_1$. If $wx_1\not\in E(G)$, then $w\in N_2(x_1)$ and $N_2(x_1)\nsubseteq V(P_1)$, a contradiction. If $wx_1\in E(G)$, then $x_2\in N_2(x_1)$ and $N_2(x_1)\nsubseteq V(P_1)$, a contradiction. So the only possible situation is that $w\in V(P_1)$ and $wx_1\not\in E(G)$. Thus $w\in N_2(x_1)$. Furthermore, by Lemma \ref{contain}, we have $N_2(x_1)\subseteq N_{C_1}(x_1)^-$ and $w\in N_{C_1}(x_1)^-$. So $w^+\in N(x_1)$~(see Fig. \ref{2 path}).
\begin{figure}[ht]
  \begin{center}
    \includegraphics[width=0.48\textwidth]{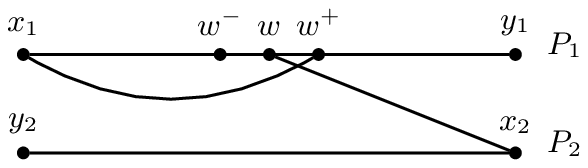}
  \end{center}
  \caption{\label{2 path}}
\end{figure}

Let $C_2=y_2P_2x_2ww^-P_1x_1w^+P_1y_1y_2$. Apparently, $C_2$ is an $id$-cycle, $Bre(C_2)=\{y_1,y_2\}$, $V(C_1)=V(C_2)$ and $def(C_2)=1<def(C_1)$. This contradicts that $C_1$ is a def-minimal $id$-cycle.

The proof is complete. \qed

\end{document}